\documentclass[reqno]{amsart}
\usepackage{amsmath}
\usepackage{amsthm}
\usepackage{amsfonts}
\usepackage{amssymb}
\usepackage{graphicx}
\usepackage{mathrsfs}
\usepackage{cite}
\usepackage{textcomp}

\title[SPHERICAL FUNCTIONS AND STOLARSKY'S INVARIANCE PRINCIPLE] 
{SPHERICAL FUNCTIONS AND STOLARSKY'S INVARIANCE PRINCIPLE}
\author[M.M. SKRIGANOV]{M.M. SKRIGANOV}

\address{St. Petersburg Department of the Steklov Mathematical Institute 
of the Russian Academy of Sciences, 
27, Fontanka, St.Petersburg 191023, Russia}

\email{maksim88138813@mail.ru}



\numberwithin{equation}{section}

\newtheorem{theorem}{Theorem}[section]
\newtheorem{lemma}{Lemma}[section]
\newtheorem{proposition}{Proposition}[section]
\newtheorem{corollary}{Corollary}[section]
\theoremstyle{remark}

\theoremstyle{remark}
\newtheorem{definition}{Definition}[section]

\def\dd{\mathrm{d}}

\def\Cc{\mathbb{C}}

\def\Ff{\mathbb{F}}

\def\Hh{\mathbb{H}}

\def\Oo{\mathbb{O}}

\def\Rr{\mathbb{R}}

\def\AAA{\mathcal{A}}
\def\BBB{\mathcal{B}}
\def\CCC{\mathcal{C}}
\def\DDD{\mathcal{D}}
\def\EEE{\mathcal{E}}
\def\FFF{\mathcal{F}}

\def\III{\mathcal{I}}

\def\MMM{\mathcal{M}}

\def\QQQ{\mathcal{Q}}

\DeclareMathOperator{\diam}{diam}

\renewcommand{\le}{\leqslant}
\renewcommand{\ge}{\geqslant}
\renewcommand{\Re}{\mathop{\mathrm{Re}}\nolimits}
\renewcommand{\Im}{\mathop{\mathrm{Im}}\nolimits}
\numberwithin{equation}{section}

\DeclareMathOperator{\Spin}{\mathrm{Spin}}

\DeclareMathOperator{\Tr}{\mathrm{Tr}}

\def\M{\mathcal M}

\def\la{\lambda}

\def\E{\mathcal E}
\def\O{\mathcal O}

\numberwithin{equation}{section}
\theoremstyle{plain}

\newcommand{\bp}{\begin{proof}}
\newcommand{\ep}{\end{proof}}
\newcommand{\bl}{\begin{lemma}}
\newcommand{\el}{\end{lemma}}
\newcommand{\bt}{\begin{theorem}}
\newcommand{\et}{\end{theorem}}
\newcommand{\bd}{\begin{definition}}
\newcommand{\ed}{\end{definition}}
\newcommand{\ba}{\begin{arrow}}
\newcommand{\ea}{\end{arrow}}
\begin{document}

%\thispagestyle{empty}

%\vspace{.5in}
%
%\begin{center}
%
%\begin{minipage}{14cm}
%\begin{center}
%\Large \bfseries
%Point distribution in two-point homogeneous~spaces
%\footnotemark[1]{}
%${}^*$
%\end{center}
%\end{minipage}
%
%\\[0.3in]
%
%\begin{tabular}{c}
%{\large\bf M. M. Skriganov} \\[6pt]
%St.Petersburg Department of\\
%Steklov Mathematical Institute\\
%Russian Academy of Sciences \\[6pt]
%E-mail: {\tt maksim88138813@mail.ru}

\begin{abstract}
In the previous paper \cite{0},	
Stolarsky's invariance principle, known for point distributions on the 
Euclidean spheres 
%$S^d$ 
\cite{33}, has been extended to the real, 
%$\Rr P^n$, 
complex, 
%$\Cc P^n$
and quaternionic 
%$\Hh P^n$ 
projective spaces and the 
octonionic 
%$\Oo P^2$ 
projective plane. 
Geometric features of these spaces as well as their  
models in terms of Jordan algebras have been used very essentially 
in the proof. In the present paper, a new pure analytic
proof of the extended Stolarsky's invariance principle is given,
%for such spaces
relying 
on the theory of spherical functions on compact symmetric
Riemannian  manifolds of rank one. 
\end{abstract}

\maketitle

\thispagestyle{empty}

\section{Introduction and main results}\label{sec1}

{\it 1.1 Introduction.} 
%Discrepancies and metrics.}
In 1973 Kenneth~B.~Stolarsky \cite{33} established the following remarkable formula for point distributions on the Euclidean spheres.
Let $S^d=\{x\in \Rr^{d+1}:\Vert x\Vert=1\}$ be the standard $d$-dimensional unit sphere in $\Rr^{d+1}$
with the geodesic (great circle) metric $\theta$ and the Lebesgue measure
$\mu$ normalized by $\mu(S^d)=1$. 
We write $\CCC(y,t)=\{x\in S^d: (x,y) > t\}$ for the spherical cap
of height $t\in [-1,1]$ centered at 
$y\in S^d$. Here we write $(\cdot,\cdot)$ and $\Vert \cdot \Vert$ for 
the inner product and the Euclidean norm in $\Rr^{d+1}$.

For an $N$-point subset $\DDD_N\subset S^d$,
the spherical cap quadratic discrepancy is defined by
\begin{equation}
\lambda^{cap}[\DDD_N]=
\int_{-1}^1\int_{S^d}\left(\,\# \{|\CCC(y,t)\cap \DDD_N\}-N\mu(\CCC(y,t))\,\right)^2\dd\mu(y)\,
\dd t.
\label{eq1.1*}
\end{equation}
We introduce the sum of pairwise Euclidean distances 
between points of $\DDD_N$
\begin{equation}
\tau [\DDD_N]=\frac 12 \sum\nolimits_{x_1, x_2\in \DDD_N} \Vert x_1-x_2 \Vert
= \sum\nolimits_{x_1, x_2\in \DDD_N} \sin\frac 12 \theta (x_1,x_2),
\label{eq1.2*}
\end{equation}
and write $\langle \tau \rangle$ for the average value of the Euclidean
distance on $S^d$,       
\begin{equation}
\langle \tau \rangle =\frac 12 \iint\nolimits_{S^d\times S^d} \Vert y_1-y_2\Vert \,
d\mu (y_1) \, \dd\mu (y_2).
%\tag{1.12}
\label{eq1.12*}
\end{equation}

The study of the quantities \eqref{eq1.1*} and \eqref{eq1.2*} falls within 
the subjects of discrepancy theory and geometry of distances,
see  \cite{2, 6} and references therein.  
It turns out that the quantities \eqref{eq1.1*} and \eqref{eq1.2*} are 
not independent and are intimately 
related by the following remarkable identity
\begin{equation}
\gamma(S^d)\lambda^{cap}[\DDD_N]+\tau [\DDD_N]=\langle \tau \rangle N^2,
\label{eq1.3**}
\end{equation}
for an arbitrary $N$-point subset $\DDD_N\subset S^d$. Here $\gamma(S^d)$
is a positive constant independent of $\DDD_N$,
\begin{equation}
\gamma(S^d)=\frac{d\,\sqrt{\pi}\,\,\Gamma(d/2)}{2\,\Gamma((d+1)/2)}\, .
%\thicksim\,\,\sqrt{\pi d/2}  \, .
\label{eq1.3***}
\end{equation}

The identity \eqref{eq1.3**} is
known in the literature as \textit{Stolarsky's invariance principle}. 
Its original proof given in \cite{33}
%of \eqref{eq1.3**} 
has been simplified in \cite{8, 11}.

In our previous paper
%the first part of this work 
\cite{0}  Stolarsky's invariance
principle \eqref{eq1.3**} has been extended to the real $\Rr P^n$,
the complex $\Cc P^n$, the quaternionic  $\Hh P^n$ projective spaces, and the 
octonionic $\Oo P^2$ projective plane.
%all these spaces.
Geometric features of such spaces as well as their  
models in terms of Jordan algebras have been used very essentially 
in the proof.
The aim of the present paper is to give an alternative pure analytic
proof relying on the theory of spherical functions.
\\

{\it 1.2 Discrepancies and metrics. $L_1$-invariance principles.}
Let us consider Stolarsky's invariance principle in a broader context.
Let $\M$ be a compact metric measure space with a fixed metric $\theta$ 
and a finite Borel measure $\mu$, normalized, for convenience, by 
\begin{equation}
\diam (\MMM,\theta)=\pi, \quad \mu (\MMM)=1,
\label{eq1.1}
\end{equation}
where 
$\diam (\EEE,\rho)=\sup \{ \rho(x_1,x_2): x_1,x_2\in \EEE\}$
denotes the diameter of a subset $\EEE\subseteq \MMM$  with respect to 
a metric $\rho$. 

We write $\BBB(y,r)=\{x\in\MMM:\theta (x,y)<r\}$ for the ball of radius $r\in \III$
centered at  $y\in \MMM$ and of volume $v(y,r)=\mu (\BBB(y,r))$. 
Here $\III = \{r=\theta(x_1,x_2): x_1,x_2\in \MMM\}$ denotes the set of all possible radii. 
If the space $\MMM$ is connected, we have $\III = [\,0,\pi\,]$.

We consider distance-invariant metric spaces for which 
the volume  of a
ball $v(r)=v(y,r)$ is independent of $y\in \MMM$, see \cite[p.~504]{24}.
The typical examples of distance-invariant spaces are homogeneous spaces 
$\MMM=G/H$ with $G$-invariant metrics
$\theta$ and  measures $\mu$.

For an $N$-point subset $\DDD_N\subset \MMM$, the ball quadratic discrepancy is 
defined by
\begin{equation}
\lambda[\xi,\DDD_N]=
\int_{\III}\int_{\MMM}\left(\,\# \{\BBB(y,r)\cap \DDD_N\}-Nv(r))\,\right)^2\,
\dd\mu(y)\, \dd \xi (r),
\label{eq1.3*}
\end{equation}
where $\xi$ is a finite measure on the set of radii $\III$.

Notice that for $S^d$ spherical caps and balls are 
related by  $\CCC(y,t)=\BBB(y,r)$, $t=\cos r$, and  the 
discrepancies \eqref{eq1.1*} and \eqref{eq1.3*} are related by
$\lambda^{cap}[\DDD_N]=\lambda[\xi^{\natural},\DDD_N]$, where  
$\dd\xi^{\natural}(r)=\sin r\,\dd r,\, r\in \III =[0,\pi]$.

The ball quadratic discrepancy \eqref{eq1.3*} can be written in the form
\begin{equation}
\lambda [\xi, \DDD_N] =
%\int\limits_\R\lambda_r [\D_N] \eta  (r) \,\dd r =
\sum\nolimits_{x_1,x_2\in \DDD_N} \la (\xi, x_1,x_2)
%\tag{1.7}
\label{eq1.7}
\end{equation}
with the kernel 
\begin{equation}
\lambda(\xi,x_1,x_2)=\int_\III\int_\MMM
\Lambda  (\BBB(y,r),x_1)\,\Lambda (\BBB(y,r),x_2)
%[\chi(B(y,r),x_1)- v(y,r)][\chi(B(y,r),x_2)- v(y,r)] 
\, \dd\mu (y)\,\dd\xi (r)\, ,
\label{eq1.6}
\end{equation}
where 
\begin{equation}
\Lambda  (\BBB(y,r),x)=\chi(\BBB(y,r),x)- v(r),
\label{eq1.6**}
\end{equation}
and $\chi(\E,\cdot)$ denotes the characteristic function of a subset $\E\subseteq\M$.

For an arbitrary metric $\rho$ on $\MMM$ we introduce 
the sum of pairwise distances  
\begin{equation}
\rho [\DDD_N] =\sum\nolimits_{x_1,x_2\in \DDD_N} \rho (x_1,x_2).
\label{eq1.10}
\end{equation}
and the average value 
%of a metric $\rho$        
\begin{equation}
\langle \rho \rangle =\int\nolimits_{\MMM\times\MMM} \rho (y_1,y_2) \,
\dd\mu (y_1) \, \dd\mu (y_2).
\label{eq1.12}
\end{equation}
We introduce the following symmetric difference metrics on the space $\M$
\begin{align}
\theta^{\Delta} (\xi, y_1,y_2) & =\frac 12\int_\III
\mu (\BBB(y_1,r)\Delta \BBB(y_2,r))
\, \dd\xi(r) \notag
\\
& =\frac 12\int_\III\int_{\MMM}\chi(\BBB(y_1,r)\Delta \BBB(y_2,r),y)
\,\dd\mu(y)\,\dd\xi(r),
\label{eq1.13}
\end{align}
where $\BBB(y_1,r)\Delta \BBB(y_2,r)=\BBB(y_1,r)\cup \BBB(y_2,r) \setminus 
\BBB(y_1,r)\cap \BBB(y_2,r)$
is the symmetric difference of the balls  $\BBB(y_1,r)$ and $\BBB(y_2,r)$.

In line with the definitions \eqref{eq1.10} and \eqref{eq1.12}, we put 
\begin{equation*}
\theta^{\Delta} [\xi,\DDD_N] =\sum\nolimits_{x_1,x_2\in \DDD_N} 
\theta^{\Delta} (\xi,x_1,x_2).
%\label{eq1.10}
\end{equation*}
and
\begin{equation*}
\langle \theta^{\Delta}(\xi) \rangle  
= \int_{\MMM\times\MMM}\theta^{\Delta}(\xi,y_1,y_2)\,\dd\mu(y_1)\,\dd\mu(y_2)\,.
\label{eq1.19}
\end{equation*}

A direct calculation leads to the following result.
\begin{proposition}\label{prop1.1}
Let a compact metric measure space $\M$ 
%with a metric $\theta$ and a measure $\mu$  
be distance-invariant, then we have
\begin{align}
\lambda (\xi,y_1,y_2)+\theta^{\Delta}(\xi ,y_1,y_2) 
= \langle 
\theta^{\Delta} (\xi)\rangle .
\label{eq1.30}
%\\
\end{align}
In particular, we have the following invariance principle
\begin{equation}
\la [\,\xi,\DDD_N\,]+\theta^{\Delta}[\,\xi , \DDD_N\,]  = \langle 
\theta^{\Delta} (\xi)\rangle \, N^2
\label{eq1.31}
\end{equation}
for an arbitrary $N$-point subset $\DDD_N\subset \MMM$.  
\end{proposition}
\begin{proof}
	%For brevity, we write $\chi(x,y,r)=\chi (\BBB(x,r),y)$. 
	In view of the symmetry
	of the metric $\theta$, we have
	\begin{equation}
		\chi (\BBB(x,r),y)=\chi (\BBB(y,r),x)=\chi_r(\theta(y,x))\, , 
		\label{eq1.100}
	\end{equation}
	where $\chi_r(\cdot)$ is the characteristic function of the segment 
	$[0,r],\, 0\le r\le\pi$.
	Therefore,
	\begin{align*}
		\chi (\BBB(y_1,r&)\Delta \BBB(y_2,r),y)
		\\
		&=
		\chi (\BBB(y_1,r),y)+\chi (\BBB(y_2,r),y)\notag
		-2\chi(\BBB(y_1,r)\cap \BBB(y_2,r),y)\, ,   
		%=\chi(y_1,y,r) +\chi 
		%(y_2,y,r) -2\chi (y_1,y,r) \chi (y_2,y,r),
		%\\
		%\label{eq1.17*}
	\end{align*}
	and 
	$\int_{\MMM}\chi (\BBB(x,r),y)\dd\mu(x)=\int_{\MMM}\chi 
	(\BBB(x,r),y)\dd\mu(y)=v(r)$.
	
	Using these relations, we obtain
	\begin{equation}\label{eq1.6*}
		\left.
		\begin{aligned}
			\lambda(\xi,x_1,x_2)=&\int_\III
			\Big (\mu  (\BBB(x_1,r)\cap \BBB(x_2,r)) -v(r)^2 \Big)\, \dd\xi 
			(r)\, ,\quad \\
			\theta^{\Delta}(\xi,y_1,y_2) 
			=&\int_{\III} \Big(v(r)-\mu (\BBB(y_1,r)\cap \BBB(y_2,r))\Big)\, 
			\dd \xi (r)\, 
			, \\
			\langle \theta^{\Delta}(\xi) \rangle  
			=&\int_{\III}\Big(v(r)-v(r)^2\Big)\, \dd\xi (r)\, .
		\end{aligned}
		%\label{eq1.6*}
		%\end{align}
		\right\}
	\end{equation}
The relations \eqref{eq1.6*} imply  \eqref{eq1.30}.
\end{proof}	

In the case of spheres $S^d$,  relations of the type 
\eqref{eq1.30} and 
\eqref{eq1.31} 
were given in \cite{33}. Their extensions to more general
metric measure spaces are given \cite[Eq.~(1.30)]{31} and 
\cite[Proposition~1.1]{0}. A probabilistic version of the invariance principle 
\eqref{eq1.31} is given \cite[Theorem~2.1]{30}.

Notice that 
\begin{equation}
\chi(\BBB(y_1,r)\Delta \BBB(y_2,r),y) = 
\vert\chi (\BBB(y_1,r),y)-\chi(\BBB(y_2,r),y)\vert\, .
\label{eq1.15}
\end{equation}
Therefore,
\begin{equation}
\theta^{\Delta} (\xi, y_1,y_2)=
\frac 12\int_\III\int_{\MMM}\vert\chi (\BBB(y_1,r),y)-\chi(\BBB(y_2,r),y)\vert
\,\dd\mu(y)\,\dd\xi(r)
\label{eq1.15*}
\end{equation}
is an $L_1$-metric. 

Recall that a metric space $\MMM$ with a metric $\rho$ is called isometrically
$L_q$-embeddable ($q=1\,\, \mbox {or}\,\, 2$), if there exists 
a mapping  $\varphi:\MMM\ni x\to \varphi(x)\in L_q$, such that
$\rho(x_1,x_2)=\Vert\varphi(x_1)-\varphi(x_2)\Vert_{L_q}$ for all $x_1$, 
$x_2\in \M$. Notice that the  $L_2$-embeddability 
is stronger and implies the $L_1$-embeddability, see~\cite[Sec.~6.3]{17}.

It follows from \eqref{eq1.15*} that the space $\MMM$ with the symmetric 
difference metrics 
$\theta^{\Delta} (\xi)$ is isometrically 
$L_1$-embeddable by the formula
\begin{equation}
\MMM \ni x\to \chi(\BBB(x,r),y)\in L_1(\MMM\times\III)\, ,
\label{eq1.150}
\end{equation}

The identity \eqref{eq1.31} can be called the $L_1$-invariance principle, 
while Stolarsky's invariance principle \eqref{eq1.3**} should be called 
the $L_2$-invariance principle, because it involves the Euclidean metric.
The identities of such a type including correspondingly $L_1$ and $L_2$
metrics could be also called weak and strong invariance principles.
\\

{\it 1.3 $L_2$-invariance principles.}
%\label{sec3}
Recall the definition and necessary  facts on two-point homogeneous spaces.
Let $G=G(\MMM)$ be the group of isometries of a metric space $\MMM$ with 
a metric $\theta$, {\it i.e.} $\theta(gx_1,gx_2)=\theta(x_1,x_2)$ for all 
$x_1$, $x_2\in \MMM$ and $g\in G$. The space $\MMM$ is called {\it 
	two-point homogeneous}, if for any two pairs of points $x_1$, $x_2$ and 
$y_1$, $y_2$ with $\theta(x_1,x_2)=\theta(y_1,y_2)$ there exists an isometry 
$g\in G$, such that $y_1=gx_1$, $y_2=gx_2$. In this case, the group $G$ is 
obviously
transitive on $\MMM$ and $\MMM=G/H$ is a homogeneous space, where
the subgroup $K\subset G$ is the stabilizer of a point $x_0\in \MMM$. 
Furthermore, the homogeneous space $\MMM$ is symmetric, {\it i.e.} for any two 
points 
$y_1$, $y_2\in \MMM$ there exists an isometry $g\in G$, such that $gy_1=y_2$, 
$gy_2=y_1$.

There is a large number of two-point homogeneous spaces.
For example, all Hamming spaces, known in the coding theory, are two-point 
homogeneous. We will consider connected spaces.
This assumption turns out to be a strong restriction. All compact connected 
two-point homogeneous spaces $\QQQ=G/H$ are 
known. By Wang's theorem they are the following, see \cite{21, 22, 
19**, 36, 37}:

(i) The $d$-dimensional Euclidean spheres 
%$S^d\subset \bR^{d+1}$, 
$S^d=SO(d+1)/SO(d)\times 
\{1\}$, $d\ge 2$, and $S^1=O(2)/O(1) \times \{ 1\}$. 

(ii) The real projective spaces $\Rr P^n=O(n+1)/O(n)\times O(1)$.

(iii) The complex projective spaces $\Cc P^n=U(n+1)/U(n)\times U(1)$.

(iv) The quaternionic projective spaces $\Hh P^n=Sp(n+1)/Sp(n)\times Sp(1)$,

(v) The octonionic projective plane $\Oo P^2=F_4/\Spin (9)$.

Here we use the standard notation from the theory of Lie groups; in 
particular,  
$F_4$ is one of the exceptional Lie groups in Cartan's classification. 

All these spaces are Riemannian symmetric manifolds of rank one. 
Geometrically, this means that all geodesic sub-manifolds in $\QQQ$ are 
one-dimensional and coincide with geodesics. From the spectral stand point, 
this also means that all 
operators on $\QQQ$ commuting with the action of the group $G$ are functions of 
the Laplace--Beltrami operator on $\QQQ$, see \cite{21, 22, 36, 37} for more 
details.

The spaces $ \Ff P^n $ as Riemannian manifolds have dimensions 
%$d$, 
\begin{equation}
d=\dim_{\Rr} \Ff P^n=nd_0, \quad d_0=\dim_{\Rr}\Ff,
\label{eq2.1} 
\end{equation}
where $d_0=1,2,4,8$ for $\Ff=\Rr$, $\Cc$, $\Hh$, $\Oo$, correspondingly.

For the spheres $S^d$ we put $d_0=d$ by definition. Projective spaces 
of dimension  $d_0$ ($n=1$) are homeomorphic to the spheres $S^{d_0}$:
$\Rr P^1\, \approx S^1, \Cc P^1\, \approx S^2,  \Hh P^1\,\approx S^4, 
\Oo P^1\, \approx S^8$.
We can conveniently agree that $d>d_0$ ($n\ge 2)$ for projective spaces,
while the equality $d=d_0$ holds only for spheres. Under this convention,
the dimensions $d=nd_0$ and $d_0$ define uniquely (up to homeomorphism) 
the corresponding homogeneous space which we denote by $\QQQ=\QQQ(d,d_0)$.
In what follows we always assume that $n=2$ if $\Ff=\Oo$, since  
projective spaces $\Oo P^n$ over octonions do not exist for $n>2$. 

We consider $\QQQ$ as a metric measure space with the metric $\theta$
and measure $\mu$ proportional to the invariant Riemannian distance and measure
on $\QQQ$. The coefficients of proportionality are defined to satisfy 
\eqref{eq1.1}.

Any space $\QQQ$ is distance-invariant and the volume of balls in the space 
is given by 
\begin{align}
v(r)&=\kappa
\int^r_0(\sin\frac{1}{2}u)^{d-1}(\cos \frac{1}{2}u)^{d_0-1}\,\dd u 
\quad r\in [\,0,\pi\,]
\notag
\\
&=\kappa \, 2^{1-d/2 -d_0/2}\int^1_{\cos r} 
(1-t)^{\frac{d}{2}-1}\,(1+t)^{\frac{d_0}{2}-1}\,\dd t,
\label{eq2.2}
\end{align}
where
$\kappa=\kappa(d,d_0)=B(d/2,d_0/2)^{-1}$;
$B(a,b)=\Gamma(a)\Gamma(b)/\Gamma(a+b)$ and 
$\Gamma(a)$ are the beta and gamma functions. 
Equivalent forms of  \eqref{eq2.2}
can be found in the literature, see, for example, \cite[pp.~177--178]{19}, 
\cite[pp.~165--168]{22},  \cite[pp.~508--510]{24}.
For even $d_0$, the integrals \eqref{eq2.2} can be calculated explicitly that 
gives convenient expressions for $v(r)$ in the case of 
$\Cc P^n,\, \Hh P^n$ and $\Oo P^2$, see, for example,  \cite{19**}. 

The chordal metric on the spaces $\QQQ$ is defined by
\begin{equation}
\tau(x_1,x_2)=\sin \frac{1}{2}\theta(x_1,x_2)\, , \quad x_1,x_2\in \QQQ.
\label{eq2.4}
\end{equation}
The formula \eqref{eq2.4} defines a metric because the function 
$\varphi(\theta)=\sin \theta/2$, $0\le \theta\le \pi$, 
is concave, increasing, and  $\varphi(0)=0$, that implies the triangle 
inequality, see \cite[Lemma~9.0.2]{17}.
For the sphere $S^d$ we have
\begin{equation}
\tau(x_1,x_2)=\sin \frac12\theta(x_1,x_2)=\frac{1}{2}\,\Vert x_1-x_2\Vert,\, 
\quad x_1,x_2\in S^d.
\label{eq2.6*}
\end{equation}
\begin{lemma}\label{lm3.1}
The space $\QQQ=\QQQ(d,d_0), \, d=nd_0,$ can be embedded into the unit sphere
\begin{equation}
\Pi:\QQQ\ni x\to \Pi(x)\in S^{m-1}\subset \Rr^m, \quad 
m=\frac{1}{2}(n+1)(d+2)-1,
%\tag{2.6}
\label{eq2.6}
\end{equation}
such that
\begin{equation}
\tau(x_1,x_2)=\left(\frac{d}{2(d+d_0)}\right)^{1/2}\|\Pi(x_1)-\Pi(x_2)\|, \quad 
x_1,x_2\in \QQQ,
%\tag{2.7}
\label{eq2.7}
\end{equation}
where $\|\cdot \|$ is the Euclidean norm in $\Rr^{m}$.
\end{lemma}
Hence, the metric $\tau(x_1,x_2)$ is proportional to the Euclidean length of a 
segment
joining the corresponding points $\Pi(x_1)$ and $\Pi(x_2)$ on the unit sphere. 
%and normalized by $\diam (Q(d,d_0),\tau)=1$.
The chordal metric $\tau$ on the complex projective space $\Cc P^n$ 
is known as the Fubini--Study metric. An interesting discussion of the 
properties of chordal metrics for projective spaces can be 
found in the paper \cite{14, 15}.

Lemma~1.1 will be proved in Section~2, and the embedding \eqref{eq2.6} will 
be described explicitly in terms of spherical functions on the space 
$\QQQ$. Note that the embedding \eqref{eq2.6} can be described in different 
ways, see, for example, \cite{0, 34*}.

The following general result has been established in \cite[Theorems~1.1 
and~1.2]{0}.
\begin{theorem}\label{thm3.1}
	For each space $\QQQ=\QQQ(d,d_0),\, d=nd_0$, we have the equality 
	%for the metrics $\tau$ and $\theta^{\Delta}(\xi^{\natural})$
	%the kernel \eqref{eq1.6}, symmetric difference metrics
	%\eqref{eq1.13}, and chordal metric \eqref{eq2.4} are related by
	\begin{equation}
	\tau(x_1,x_2)=\gamma(\QQQ)\,\, \theta^{\Delta}(\xi^{\natural} ,x_1,x_2). 
	%x_1,x_2\in \QQQ\, .
	\label{eq2.8}
	\end{equation}
	where $\dd\xi^{\natural}(r)=\sin r\,\dd r$, $r\in [0,\pi]$ and
	\begin{equation}
	\gamma(\QQQ)
	=\frac{\sqrt{\pi}}{4}\,(d+d_0)\,
	\frac{\Gamma(d_0/2)}{\Gamma((d_0+1)/2)}
	=\frac{n+1}{2}\,\gamma(S^{d_0})\, ,
	\label{eq1.33*}
	\end{equation}
	where $\gamma(S^{d_0})$ is defined by \eqref{eq1.3***}. Therefore, 
	\begin{equation*}
		\left.
		\begin{aligned}
			& \gamma (\Rr P^n)  = \frac{n+1}{2}\, \gamma (S^{1}) = 
			\frac{\pi}{4} \, 
			(n+1)\, ,
			%\notag
			\\
			& \gamma (\Cc P^n)=  \frac{n+1}{2}\, \gamma (S^{2}) = n+1\, ,
			%\notag
			\\
			& \gamma (\Hh P^n) =  \frac{n+1}{2}\, \gamma (S^{4}) = 
			\frac{4}{3}\, 
			(n+1)\, ,
			%\notag
			\\
			& \gamma (\Oo P^2) = \,\,\frac{3}{2}\, \gamma (S^{8}) = 
			\frac{192}{35}\, .
			%\notag
		\end{aligned}
		\label{eq1.34}
		\right\}
	\end{equation*}
	\end{theorem}
Comparing Theorem~1.1 with Proposition~1.1, we arrive to the following.	
	\begin{corollary}\label{cor}
	We have the following $L_2$-invariance principle
	\begin{equation}
		\gamma(\QQQ)\,\lambda [\xi^{\natural},\DDD_N]+\tau [\DDD_N]=\langle 
		\tau\rangle 
		N^2
		\label{eq2.10}
	\end{equation}
for an arbitrary $N$-point subset $\DDD_N\subset \QQQ$.
\end{corollary}

The constant $\gamma(\QQQ)$ has the following geometric 
interpretation
\begin{equation}
\gamma(Q)=\frac{\langle \tau\rangle}{\langle 
	\theta^{\Delta}(\xi^{\natural})\rangle}=\frac{\diam 
	(\QQQ,\tau)}{\diam(\QQQ,\,\theta^{\Delta}(\xi^{\natural}))}\, .
\label{eq2.9}
\end{equation}
Indeed, it suffices to calculate the average values \eqref{eq1.12} of both 
metrics 
in \eqref{eq2.8} to obtain 
the first equality in \eqref{eq2.9}. Similarly, writing \eqref{eq2.8} for any 
pair of antipodal points $x_1$, $x_2$, $\theta(x_1,x_2) = \pi$, we obtain the 
second 
equality in \eqref{eq2.9}. 
The average value $\langle \tau \rangle$ of the chordal metric $\tau$ can be 
easily calculated with the help of the formulas \eqref{eq1.12} and 
\eqref{eq2.2}:
\begin{equation}
	\langle \tau\rangle 
	= B(d/2, d_0 /2)^{-1}\, B((d+1)/2, d_0 /2)\, .
	\label{eq0aaa}
\end{equation}

In the case of spheres $S^d$,
the identity \eqref{eq2.10} coincides with \eqref{eq1.3**}. 
The identity \eqref{eq2.10} can be thought of as an extension of
Stolarsky's invariance principle to all projective spaces.

Applications of $L_1$- and $L_2$-invariance principles and similar identities 
to the discrepancy theory, geometry of distances, information and probability 
theory 
%on $S^d$
have been given in many papers, see, for example, 
\cite{2, 3, 3*, 5, 6, 8, 8a, 8b, 11, 30, 31, 0, 33}.

%%%%%%%%%%%%%%%%%%%%%%%%%%%%%%%%%%%%%%%%%%%%%%%%%%%

%%%%%%%%%%%%%%%%%%%%%%%%%%%%%%%%%%%%%%%%%%%%%%
The above discussion raises the following {\it open questions}: 

- Are there measures $\xi$ on the set of radii for spaces $Q$ 
(for spheres $S^d$, say) 
other than the measure $\xi^{\natural}$ such that the 
corresponding symmetric difference metrics $\theta^{\Delta}(\xi)$ are the 
$L_2$-metrics?

- Are there compact measure metric spaces other than spheres $S^d$ and 
projective spaces $\Ff P^n$ for which the $L_2$-invariance principle is also 
true?
\\

{\it 1.4 Proof of Theorem~1.2.}
In the present paper we use the theory of 
spherical functions to prove the following result.
\begin{theorem}\label{prop3.1}
	%For any space $Q=Q(d,d_0)$, 
	The equality \eqref{eq2.8} is equivalent to
	the following series of formulas for Jacobi polynomials
	\begin{align}
		\int^1_{-1}
		\left(P^{(d/2, d_0/2)}_{l}(t)\right)^2& \, \left(1-t\right)^{d}
		\left(1+t\right)^{d_0}\, \dd t
		\notag
		\\ 
		= &\,\,\frac{2^{d+d_0+1}\, (1/2)_{l-1}}{(l!)^2}\,B(d+1, d_0 +1)\, 
		T_{l}(d/2, d_0/2)\,
		\label{eq1.33} 
		%\quad l\ge 1\, ,
	\end{align}
	for all $l\ge 0$, where
	\begin{align}
		T_{l}(d/2, d_0/2)
		%\notag
		%\\
		=\frac{\Gamma(d/2+l+1)\,\Gamma(d_0/2+l+1)\,\Gamma(d/2+d_0/2+3/2))}
		{\Gamma(d/2+1)\,\Gamma(d_0/2+1)\,\Gamma(d/2+d_0/2+3/2+l)} \, .
		%\, B(d+1, d_0 +1),
		\label{eq1.34}
	\end{align} 
	%for all $l\ge 1$.
\end{theorem}
%Theorem~1.2 will be proved in Section~2.
Here $P^{(\alpha,\beta)}_l(t),\, t\in [-1,1],\, \alpha > -1,\, \beta > -1,$ are 
the 
standard Jacobi polynomials of degree $l$ normalized by 
\begin{equation}
	%\vert P^{(\alpha,\beta)}_n(t)\vert \le 
	P^{(\alpha,\beta)}_l(1) =
	\binom{\alpha +l}{l} =\frac{\Gamma (\alpha +l+1)}{\Gamma (l+1) 
	\Gamma(\alpha +1)}\, . 
	\label{eq8.23}
\end{equation}
The polynomials  $P^{(\alpha,\beta)}_l$ can be given by Rodrigues' formula 
\begin{equation}
	P^{(\alpha,\beta)}_l(t) =
	\frac{(-1)^l}{2^l l!} (1-t)^{-\alpha} (1+t)^{-\beta}
	\frac{\dd^l}{\dd t^l} 
	\left\{ (1-t)^{n+\alpha} (1+t)^{n+\beta} \right\}.
	%\tag{9.11}
	\label{eq9.11}
\end{equation} 

Notice that 
$\vert P^{(\alpha,\beta)}_l(t)\vert \le P^{(\alpha,\beta)}_l(1)$ 
for $t\in [-1,1]$. 
Recall that $\{P^{(\alpha,\beta)}_l, l\ge 0\}$ form a complete 
orthogonal system in 
$L_2$ on the segment $[-1,1]$ with the weight $(1-t)^{\alpha}(1+t)^{\beta}$ 
and the following orthogonality relations 
\begin{align}
	& \int\limits^{\pi}_{0}P^{(\alpha,\beta)}_l (\cos u)
	P^{(\alpha,\beta)}_{l'}(\cos u) (\sin \frac12 u)^{2\alpha +1} 
	(\cos \frac 12 u)^{2\beta +1}\, du
	%\notag
	\\
	& = 
	2^{-\alpha-\beta-1} \int\limits^1_{-1}
	P^{(\alpha,\beta)}_l(t) P^{(\alpha,\beta)}_{l'} (t)
	(1-t)^{\alpha}(1+t)^{\beta} \, dt =M_l(\alpha,\beta)^{-1} \delta_{ll'}, 
	\label{eq8.25}
\end{align}                                               
where $\delta_{ll'}$ is Kronecker's symbol, $M_0=B(\alpha +1,\beta +1)^{-1}$ 
and 
\begin{equation}
	M_l(\alpha,\beta)=(2l+\alpha+\beta+1) 
	\frac{\Gamma (l+1)\Gamma (l+\alpha+\beta+1)}{\Gamma(l+\alpha+1)
		\Gamma (l+\beta+1)}, 
	%\simeq l, 
	\quad l\ge 1.
	%\tag{8.26}
	\label{eq8.26**}
\end{equation}

All necessary facts about Jacobi polynomials
can be found in \cite{1*, 34}.
We also use the notation
\begin{equation}
	(a)_0 = 1,\quad (a)_k = a(a+1)\dots (a+k-1)=\frac{\Gamma(a +k)}{\Gamma(a)}
	\label{eq8.26*}
\end{equation}
for the rising factorial powers and
%(Pochhammer's symbol).
\begin{equation}
	\langle a \rangle_0 = 1,\quad \langle a \rangle_k = a(a-1)\dots 
	(a-k+1)=(-1)^k\,(-a)_k\,  
	\label{eq8.26***}
\end{equation}
for the falling factorial powers. 

Theorem~1.2 reduces the proof of Theorem~1.1 to the proof of the 
formulas \eqref{eq1.33}. Perhaps such formulas are known but I 
could not find them in the literature. Not much is known about
the integrals 
$\int^1_{-1}\left(P^{(\alpha, \beta)}_{l}(t)\right)^2 
(1-t)^{\nu}
(1+t)^{\sigma}\, \dd t $ for general Jacobi polynomials $P^{(\alpha, \beta)}$.
Only for spheres $S^d$ Jacobi polynomials $P^{(d/2, d/2)}_{l}$ 
%with equal parameters 
coincide (up to constant factors) with 
Gegenbauer polynomials, and in this case very general formulas for such 
integrals
%weighted $L_2$-norms of Gegenbauer polynomials
are given in \cite{11*}.

In the present paper we will prove the following statement.
\begin{lemma}\label{lem3.1}
	For all $l\ge 0$, $\Re\alpha >-1/2$ and $\,\Re\beta >-1/2$, we have
	\begin{align}
		\int^1_{-1}\left(P^{(\alpha, \beta)}_{l}(t)\right)^2& 
		(1-t)^{2\alpha}
		(1+t)^{2\beta}\, \dd t 
		\notag
		\\
		=&\,\, \frac{2^{2\alpha+2\beta+1}\, (1/2)_{l}}{(l!)^2}\,B(2\alpha+1, 
		2\beta +1)\,T_l(\alpha, \beta),
		\label{eq01}
	\end{align}
	where
	\begin{align}
		T_l(\alpha, \beta)=&
		\,\,\frac{(\alpha +1)_l\, (\beta +1)_l}{(\alpha +\beta +3/2)_l}
		\notag
		\\
		=&\,\,\frac{\Gamma(\alpha+l+1)\,\Gamma(\beta+l+1)\,
			\Gamma(\alpha+\beta+3/2))}
		{\Gamma(\alpha+1)\,\Gamma(\beta+1)\,\Gamma(\alpha +\beta +3/2+l)}  
		\label{eq01*}
	\end{align} 
	is a rational function of $\alpha$ and $\beta$.
\end{lemma}

The integral \eqref{eq01} converges for $\Re\alpha >-1/2$ and $\,\Re\beta 
>-1/2$,
and represents in this region 
a holomorphic function of two complex variables.  
The equality \eqref{eq01} defines an analytic continuation of 
the integral \eqref{eq01} to $\alpha\in\Cc$ and $\beta\in\Cc$.

For $\alpha=d/2$, $\beta=d_0/2$ and $l$ replaced with $l-1$,
the equality \eqref{eq01} coincides with \eqref{eq1.33}. This proves 
Theorem~1.1. 

Lemma~1.2 will be proved in Section~3. 
The crucial point in the proof is Watson's theorem on the value of 
hypergeometric series $_3F_2(1)$.

\section{Spherical functions. Proofs of Lemma~1.1 and Theorem~1.2}\label{sec4}

{\it 2.1. Invariant kernels and spherical functions.} The general theory of 
spherical 
functions on homogeneous spaces can be found 
in \cite{21, 22, 35, 37}.
The homogeneous spaces $\QQQ$ of interest to us belong to the class of 
so-called commutative spaces and symmetric Gelfand pairs. In this case the 
theory 
becomes significantly simpler. For Euclidean spheres $S^d$ this theory is well 
known, see, for example, \cite{17*, 24**}.  
However, the theory 
of spherical functions on general spaces $\QQQ$ is probably not commonly 
known. In this section we describe the basic facts about spherical functions on 
spaces $\QQQ$ in a form convenient for our purposes.

Let us consider the quasi-regular representation 
$U(g)f(x)=f(g^{-1}x),\, f\in L_2(\QQQ)$, $x \in\QQQ,\, g\in G$, and its 
decomposition 
into the orthogonal sum 
\begin{equation}\label{eqdecom}
	U(g)=\widehat{\bigoplus} _{l\ge 0}\,\, U_l(g), \quad 
	L_2(\QQQ)=\widehat{\bigoplus} _{l\ge 0}\,\, V_l\, ,
\end{equation} 
of irreducible representations $U_l(g)$ in mutually orthogonal subspaces $V_l$ 
of dimensions $m_l<\infty$.

Let $\AAA$ denote the algebra of Hilbert--Schmidt operators $K$ in $L_2(\QQQ)$ 
commuting with the action of the group $G$: $KU(g)=U(g)K, g\in G$. Each 
$K\in\AAA$ is an integral operator
$Kf(x)=\int_{\QQQ} \, K(x,y)\, f(y)\, \dd \mu (y)$ with the {\it invariant 
kernel}: 
\begin{equation}\label{eqC*}
	K(gx_1,gx_2)=K(x_1,x_2), \,\, x_1, x_2\in \QQQ,\,\, g \in G,
\end{equation}
and the Hilbert-Schmidt norm $\Vert K\Vert_{HS}$ defined by
\begin{align}\label{eqC}
\Vert K\Vert^2_{HS}=&\Tr\, K K^*
\notag
\\
=&\int_{\QQQ\times \QQQ} \, |K(x,y)|^2\, \dd \mu (x)  \dd 
\mu (y)=\int_{\QQQ} \, |K(x,y)|^2\, \dd \mu (x) <\infty,
\end{align}
where $\Tr$ denotes the trace of an operator, and the second integral is 
independent of $y$ in view of \eqref{eqC*}. 

Since the space $\QQQ$ is two-point homogeneous, the condition \eqref{eqC*}
implies that the kernel $K(x_1,x_2)$ depends only on the distance $\theta 
(x_1,x_2)$, and can be written as
\begin{equation}\label{eqC1}
	K(x_1,x_2)=K(\theta (x_1,x_2))=k(\cos \theta (x_1,x_2)), \,\, x_1, 
	x_2\in \QQQ,
\end{equation}
with functions $K(u),\, u\in [0,\pi],$ and $k(t),\, t\in [-1,1]$. The 
cosine 
is 
presented here for convenience in some further calculations. It is useful to 
keep in 
mind that the formula \eqref{eqC1} can be also written as 
\begin{equation}\label{eqC1*}
	K(x_1,x_2)=K(\theta (x,x_0))=k(\cos \theta (x,x_0)), 
\end{equation}
where $x_0\in \QQQ$ is the fixed point of the subgroup $H,\,  x_1=g_1x_0,\, 
x_2=g_2x_0,\, x=g^{-1}_2g_1x_0,\,\,g_1,\, g_2\in G$, and
$K(hx,x_0)=K(x,x_0),\, h\in H$. Therefore, invariant 
kernels can be thought of as functions on the double co-sets $H\setminus G/H$.

In terms of the function $K(\cdot)$ and $k(\cdot)$, the Hilbert-Schmidt norm 
\eqref{eqC} 
takes the form 
\begin{align}\label{eqC2}
	\Vert K\Vert_{HS}^2&=\int^{\pi}_0\,|K(u)|^2\,\dd v(u)
\notag
\\
	=&\kappa
	\int^{\pi}_0\, |K(u)|^2(\sin\frac{1}{2}u)^{d-1}(\cos 
	\frac{1}{2}u)^{d_0-1}\,\dd u 
	%\quad r\in [\,0,\pi\,],
	%\label{eqC1*}
\notag
\\
%\notag
%\\
	=&\kappa\, 2^{1-d/2-d_0/2}\,\int_{-1}^1\,|k(z)|^2 
	\,(1-z)^{\frac{d}{2}-1}\, (1+z)^{\frac{d_0}{2}-1}\, 
	\, \dd z , 
	%<\infty,
\end{align} 
here $v(\cdot)$ is the volume function \eqref{eq2.2}. 

We conclude from \eqref{eqC*} and \eqref{eqC1} that for $K\in\AAA$ its kernel 
is symmetric, 
$K(x_1,x_2)=K(x_2,x_1)$, the value $K(x,x)=k(1)$ is independent 
of $x\in \QQQ$, and if an operator $K$ is self-adjoint, then its kernel is 
real-valued. 

It follows from \eqref{eqC*} and \eqref{eqC1} that the algebra $\AAA$ is 
commutative. Indeed,
\begin{align*}\label{eqC3}
	(K_1K_2)(x_1,x_2)&=\int_{\QQQ} \,K_1(x_1,x)K_2(x,x_2) 
	\dd \mu(x) 
	\notag
	\\
	=& \int_{\QQQ} \,K_2(x_2,x)K_1(x,x_1) 
	\dd \mu(x) =(K_2K_1)(x_2,x_1)=(K_2K_1)(x_1,x_2).
\end{align*}
Therefore, the decomposition \eqref{eqdecom} is 
multiplicity-free, that is any two representations $U_l$ and $U_{l'}$, $l\ne 
l'$, are non-equivalent, because otherwise the algebras $\AAA$ 
could not be commutative. 

Let $P_l$ denote orthogonal projectors in $L_2(\QQQ)$ onto the subspaces $V_l$ 
in 
\eqref{eqdecom}, 
\begin{equation}\label{eqproj}
	P_l^*=P_l\, ,\quad P_l\, P_{l'} 
	=\delta_{l,l'}\, P_l\, ,\quad
	\sum\nolimits_{l\ge 0}P_l = I\, , 
\end{equation}
where $\delta_{l,l'}$ is Kronecker's symbol and $I$ is the identity operator in 
$L_2(\QQQ)$. 
By Schur's lemma,  we have for $K\in\AAA$
\begin{equation}\label{eqC9}
	P_l\,K\,P_{l'}=\delta_{l,l'}\,c_l(K)\,P_l, ,
\end{equation}
where $c_l(K)$ is a constant. Calculating the trace of both sides of the 
equality 
\eqref{eqC9}, we find $c_l(K)=m_l^{-1}\Tr KP_l$. Therefore, we have the 
expansions 
\begin{equation}\label{eqC9*}
K=\sum\nolimits_{l,l'\ge 0}	P_l\,K\,P_{l'}
%=\sum\nolimits_{l\ge 0}	P_l\,K\,P_{l}
=\sum\nolimits_{l\ge 0}c_l(K)\,P_{l} ,
\end{equation}
and
\begin{equation}\label{eqC9**}
	K_1\,K_2=\sum\nolimits_{l\ge 0}c_l(K_1)\,c_l(K_2)\,P_{l} ,
\end{equation}
for $K_1, K_2\in\AAA$. It follows from \eqref{eqC9**} with $K_1=K$ and 
$K_2=K^*$ that
\begin{equation}\label{eqC9***}
	||K||_{HS}^2=\sum\nolimits_{l\ge 0}m_l\,|c_l(K)|^2\,<\infty ,
\end{equation}

The equality \eqref{eqC9***} implies that the series 
\eqref{eqC9**} converges in the norm \eqref{eqC}, and the 
series \eqref{eqC9**} converges in the norm \eqref{eqC} for the subclass of 
nuclear operators. 

Since $V_l$ are invariant subspaces, $P_l\in\AAA$, their 
kernels $P_l(\cdot,\cdot)$ are symmetric and real-valued, and can be written as
follows
\begin{equation}\label{eqC4}
	P_l(x_1,x_2)=p_l(\cos 
	\theta(x_1,x_2))=\sum\nolimits_1^{m_l}\,\psi_{l,j}(x_1)\,\psi_{l,j}(x_2),
\end{equation} 
where $\{\psi_{l,j} (\cdot)\}_1^{m_l}$ is an orthonormal and real-valued basis 
in $V_l$. Hence,  subspace $V_l$ and irreducible representations $U_l$
in \eqref{eqdecom} can be thought of as defined over the field of reals, this 
means that all representations  $U_l$ in \eqref{eqdecom} are of the real type.

Using \eqref{eqC4}, we obtain the formulas
\begin{equation}\label{eqC5}
	\Vert P_l\Vert^2_{HS}=m_l, \quad    \Tr P_l=p_l(1)=m_l >0.
\end{equation}  
Furthermore, 
\begin{equation}\label{eqC6}
P_l(x,x)=p_l(1)=\sum\nolimits_1^{m_l}\,\psi_{l,j}(x)^2	.
\end{equation} 
is independent of $x\in \QQQ$. 
Applying Cauchy-Schwartz inequality to \eqref{eqC4} and taking \eqref{eqC6} 
into account, we obtain the bound
\begin{equation}\label{eqC8}
	|P_l(x_1,x_2)|=|p_l(\cos\theta(x_1,x_2))|\le p_l(1).
\end{equation}

It follows from \eqref{eqC6} and \eqref{eqC5} that the mapping 
\begin{equation}\label{eqC7}
\Pi_l:\,\QQQ\ni x\to 
(m_l^{-1/2}\psi_{l,1}(x)\dots m_l^{-1/2}\psi_{l,m_l}(x))\in 
S^{m_l-1}\subset\Rr^{m_l}	
\end{equation}
defines an embedding of the space $\QQQ$ into the unite sphere in $\Rr^{m_l}$.

By definition the (zonal) {\it spherical function} are kernels of the operators 
$\Phi_l=m_l^{-1}P_l$:
\begin{equation}\label{eqC8}
\Phi_l(x_1,x_2)=	
\phi_l(\cos\theta(x_1,x_2))=\frac{p_l(\cos\theta(x_1,x_2))}{p_l(1)}.
\end{equation}
From \eqref{eqC6} and \eqref{eqC8} we conclude that 
$|\phi_l(\cos\theta(x_1,x_2))|\le\phi_l(1)=1$. Comparing 
\eqref{eqC5}, \eqref{eqC6} and \eqref{eqC8}, we find the formulas for dimensions
\begin{align}\label{eqC8*}
m_l&=||\Phi_l||_{HS}^{-2}
%\notag
%\\
=\left(\kappa
\int^{\pi}_0\, |\phi_l(\cos u)|^2(\sin\frac{1}{2}u)^{d-1}(\cos 
\frac{1}{2}u)^{d_0-1}\,\dd u  \right)^{-1}
\notag
\\
&=\left(\kappa\,  2^{1-d/2-d_0/2}\int_{-1}^1\,|\phi_l(t)|^2 
\,(1-t)^{\frac{d}{2}-1}\, (1+t)^{\frac{d_0}{2}-1}\, 
\, \dd t \right)^{-1}.	
	\end{align}
%with the constant $a=B(d/2,d_0/2)^{-1}\, 2^{1-d/2-d_0/2}$.

In terms of spherical functions the formulas \eqref{eqC9*}, \eqref{eqC9**} and 
\eqref{eqC9***} take the form
\begin{equation}\label{eqC10*}
K(x_1,x_2)=	k(\cos\theta(x_1,x_2))
	=\sum\nolimits_{l\ge 0}c_l(K)\,m_l\,\phi_l(\cos\theta(x_1,x_2)) ,
\end{equation}
where
\begin{align}\label{eqC10**}
	c_l(K)&=\Tr K\Phi_l=\int_Q K(x,x_0)\,\Phi (x,x_0)\,\dd \mu (x)
	\notag
	\\
	&=\kappa
	\int^{\pi}_0\, K(u)\,\phi_l(\cos u)\,(\sin\frac{1}{2}u)^{d-1}(\cos 
	\frac{1}{2}u)^{d_0-1}\,\dd u  
	\notag
	\\
	&=\kappa\,  2^{1-d/2-d_0/2}\int_{-1}^1\,k(t)\,\phi_l(t) 
	\,(1-t)^{\frac{d}{2}-1}\, (1+t)^{\frac{d_0}{2}-1}\, 
	\, \dd t ,	
\end{align}
and 
\begin{align}\label{eqC10***}
(K_1K_2)(x_1,x_2)=&\int_Q K_1(\theta (x_1,y))\,K_2(\theta (y,x_2))\,\dd\mu (y)
\notag
\\
=&\int_Q k_1(\cos\theta (x_1,y))\,k_2(\cos\theta (y,x_2))\,\dd\mu (y)
	\notag
	\\
	=&\sum\nolimits_{l\ge 0}c_l(K_1)\,c_l(K_2)\,m_l\,\phi_{l}(\cos\theta 
	(x_1,x_2)) ,
\end{align} 
and  
\begin{equation}\label{eqC10****}
	||K||_{HS}^2=\sum\nolimits_{l\ge 0}m_l\,|c_l(K)|^2 <\infty \, ,
\end{equation}

The facts listed above are valid for all compact two-point homogeneous spaces. 
Since the spaces $Q$ are also symmetric Riemannian manifolds of rank one, the 
invariant kernels 
$p_l(\cos\theta (x,x_0))$ are eigenfunctions of the radial part of the 
Laplace--Beltrami operator on $Q$ (in the spherical coordinates centered at 
$x_0$). This leads to the following explicit formula for spherical functions
\begin{equation}
%\phi_l(x_1,x_2) =
\Phi (x_1,x_2)=\phi_l(\cos\theta (x_1,x_2)) = \frac{P^{(\frac{d}{2} 
-1,\frac{d_0}{2} -1)}_l 
(\cos\theta(x_1,x_2))}
{P^{(\frac{d}{2} -1,\frac{d_0}{2} -1)}_l (1)}, \quad l\ge 0.
%\,\, x_1, x_2 \in Q,
\label{eq4.1}
\end{equation}
where $P^{(\alpha,\beta)}_n (t),\, t\in [-1,1],$ are  Jacobi polynomials 
\eqref{eq9.11}.
%of degree $n$.
For more details, we refer to \cite[p.~178]{19}, 
\cite[Chap.~V, Theorem~4.5]{22},
\cite[pp.~514--512, 543--544]{24}, 
\cite[Chapters 2 and 17]{35}: 
\cite[Theorem~11.4.21]{37}.

From \eqref{eq8.23} and \eqref{eq8.26**} we obtain
\begin{equation}
	P^{(\frac{d}{2} -1,\frac{d_0}{2} -1)}_l(1) =
	%\binom{n-1+d/2}{n} =
	\frac{\Gamma (l+d/2)}{\Gamma (l+1) \Gamma(1+d/2)}\, , 
	\label{eq8.23**}
\end{equation}
and $M_l=M_l(\frac{d}{2} -1,\frac{d_0}{2} -1)$, where 
$M_0=B(d/2,d_0/2)^{-1}$ and 
\begin{equation}
	M_l=
	%M_l(d/2-1,d_0/2-1)=
	(2l - 1+(d+d_0)/2 ) 
	\frac{\Gamma (l+1)\Gamma (l-1+(d+d_0)/2)}{\Gamma(l+d/2)
		\Gamma (l+d_0/2)}\,  ,
	%\simeq l, 
	\; \; l\ge 1,
	\label{eq8.26***}
\end{equation}
%%%%%%%%%%%%%%%%%%%%%%%%%%%%%%%%%%%
Substituting \eqref{eq4.1} 
into \eqref{eqC8*} and using \eqref{eq8.23**} and \eqref{eq8.26***}, we 
obtain the following explicit formulas for dimensions of irreducible 
representations \eqref{eqdecom} : $m_0=1$ and
\begin{align}\label{eq8.27*}
	m_l&=M_l\,\kappa ^{-1} \,
	(P^{(\frac{d}{2} -1,\frac{d_0}{2} -1)}(1))^2 
\notag
\\	
	=&(2l-1+(d+d_0)/2)\frac{\Gamma(l-1+(d+d_0)/2)\Gamma(l+d/2)\Gamma(d_0/2)}
	{\Gamma((d+d_0)/2)\Gamma(l+d_0/2)\Gamma(d/2)\Gamma(l+1)},\quad l\ge 1.
\end{align}

{\it 2.2. Spherical functions and metrics.}
In the following Lemma~2.1 we will describe the spherical function expansions 
for 
the chordal and symmetric difference metrics. Originally these expansions have 
been established in \cite[Lemma~4.1]{0} and \cite[Theorema~4.1(ii)]{31}. For 
completeness, we give a short proof of these results.
\begin{lemma}\label{lm2.1}
(i)	For the chordal metric \eqref{eq2.4}, we have	
\begin{equation}
\tau(x_1,x_2)=\frac12\sum\nolimits_{l\ge 1}\, M_l\,C_l\,\left[\, 
1-\phi_l(\cos\theta(x_1,x_2))\,\right],
\label{eq9.16**}
\end{equation}
where 
\begin{align}
C_l& = B((d+1)/2, l+d_0/2)\, 
\Gamma (l+1)^{-1}\, (1/2)_{l-1}\, P^{(\frac{d}{2} -1,\frac{d_0}{2} -1)}_l 
(1)\,.
\label{eq8.26*}
\end{align}

(ii) For the symmetric difference metrics \eqref{eq1.13}, we have 
\begin{equation}
\theta^{\Delta}(\xi, x_1,x_2) =
%B(d/2,d_0/2)^{-1}
\kappa \, \sum\nolimits_{l\ge 1}\, l^{-2}\, M_l\, A_l(\xi)\, 
\left[\, 1-\phi_l(\cos\theta(x_1,x_2))\,\right],
\label{eq9.16}
\end{equation}
where 
\begin{equation}
A_l(\xi)=\int^{\pi}_{0}
\left\{ P^{(\frac{d}{2},\frac{d_0}{2})}_{l-1} (\cos r)\right\}^2  
(\sin\frac 12 r)^{2d}(\cos \frac12 r)^{2d_0} 
\, \dd \xi (r). 
\label{eq9.16*}
\end{equation}

The series \eqref{eq9.16**} and \eqref{eq9.16} converge absolutely and 
uniformly.
\end{lemma}
\begin{proof}[Proof.] {\it (i)} Let us consider the expansion \eqref{eqC10*} 
for the 
chordal metric \eqref{eq2.4}. Since
\begin{equation}
	\tau(x_1,x_2)=\sin \frac{1}{2}\theta(x_1,x_2).
	= \sqrt{\frac{1-\cos (x_1, x_2)}{2}}\, , 
	%\quad x_1,x_2\in Q(d,d_0),
	%\tag{2.4}
	\label{eq2.4*}
\end{equation}
we put $k(t)=\sqrt{(1-t)/2}$ in the formula \eqref{eqC10**}. This gives 
\begin{equation}\label{eqA10*}
	c_l(K)=
	\kappa\,  2^{1/2-d/2-d_0/2} \left(P^{(\frac{d}{2} -1,\frac{d_0}{2} 
	-1)}_l(1) 
	\right)^{-1}\, I_l(K)\, ,
\end{equation}
where
\begin{align}\label{eqA10**}
	I_l(K)&=\int_{-1}^1 \,P^{(\frac{d}{2} -1,\frac{d_0}{2} -1)}_1(t) 
	\,(1-t)^{\frac{d}{2}-\frac{1}{2}}\, 
	(1+t)^{\frac{d_0}{2}-1}\, \dd t 
	\notag
	\\
	&=2^{d/2+d_0/2-1/2}	(l!)^{-1}\,(-1/2)_l\, B(d/2+1/2,d_0/2+l),
\end{align}
The formula \eqref{eqA10**} can be found in the tables \cite[Sec.7.391, 
Eq.(4)]{01} or derived directly, using Rodrigues' formula \eqref{eq9.11} and 
integrating $l$ times by parts.

Notice that the symbol $(-1/2)_l$ in \eqref{eqA10**} takes the values 
$(-1/2)_0=1$ and $(-1/2)_l=-1/2(1/2)_{l-1}$ for $l\ge 1$.
Using \eqref{eqA10**}, \eqref{eqA10*}, \eqref{eq8.27*} and \eqref{eq8.26***},
we find that 
$
	m_l\,c_l(k)=-1/2\,M_l\,C_l
	%\quad l\ge 1, 
	$
%\end{equation}
for $l\ge 1$,
where $C_l$ are given in \eqref{eq8.26*}. Therefore, the expansion 
\eqref{eqC10*} takes the form
\begin{equation}
	\tau(x_1,x_2)=c_0-\frac12\sum\nolimits_{l\ge 1}\, M_l\,C_l\, 
	\phi_l(\cos\theta(x_1,x_2))\, ,
	\label{eqA.16}
\end{equation}
We put here $x_1=x_2$ to obtain $c_0=1/2\sum\nolimits_{l\ge 1}\, M_l\,C_l$.
Substituting this equality into \eqref{eqA.16}, we arrive to the expansion 
\eqref{eq9.16**}.

Applying Stirling's approximation to the gamma functions in $M_l$ and $C_l$, we 
observe that the terms in \eqref{eq9.16**} are of the order $O(l^{-2})$.
Therefore, the series \eqref{eq9.16**} converges absolutely and uniformly.

{\it (ii)}  Let us consider the expansion \eqref{eqC10*} 
for the symmetric difference metric \eqref{eq1.13}. We have
\begin{equation}
\theta^{\Delta}(\xi,x_1,x_2) 
=\int_0^{\pi} \Big(v(r)-\mu (\BBB(y_1,r)\cap \BBB(y_2,r))\Big)\, 
\dd \xi (r)\, , 
\label{eqA.16*}	
	\end{equation}
see \eqref{eq1.6*}. In view of \eqref{eq1.100} the term 
$\mu (\BBB(y_1,r)\cap \BBB(y_2,r))$ can be written as
\begin{equation}
\mu (\BBB(y_1,r)\cap 
\BBB(y_2,r))=\int_{\QQQ}\chi_r(\theta(x_1,y))\,\chi_r(\theta(y,x_2)\,\dd \mu 
(y) , 
	\label{eq1.100*}
\end{equation}
where $\chi_r(\cdot)$ is the characteristic function of the segment 
$[0,r],\, 0\le r\le\pi$.

Let us consider the expansion \eqref{eqC10*} for the invariant kernel
$\chi_r(\theta(x_1,x_2))$. We put $K(u)=\chi_r(u),\, u\in [0,\pi],$
to calculate the corresponding coefficients \eqref{eqC10**}. We obtain
$c_0(K)=v(r)$ and
\begin{equation}\label{eqA10*}
	c_l(K)=
	\kappa\,  
	%2^{1/2-d/2-d_0/2} 
	\left(P^{(\frac{d}{2} -1,\frac{d_0}{2} 
		-1)}_l(1) 
	\right)^{-1}\, I_l(K)\, ,\quad l\ge 1,
\end{equation}
where
\begin{align}\label{eqB10}
	I_l(K)&=
	\int^{r}_0\, \,P^{(\frac{d}{2} -1,\frac{d_0}{2} 
	-1)}_1(\cos u)\,(\sin\frac{1}{2}u)^{d-1}(\cos 
	\frac{1}{2}u)^{d_0-1}\,\dd u  
	\notag
	\\
	&=2^{1-d/2-d_0/2}\int_{\cos r}^1 \,P^{(\frac{d}{2} -1,\frac{d_0}{2} 
	-1)}_1(t) 
	\,(1-t)^{\frac{d}{2}-1}\, (1+t)^{\frac{d_0}{2}-1}\, 
	\, \dd t ,	
	\notag
	\\
	&=l^{-1}\, (\sin\frac{1}{2}r)^{d}\,(\cos 
	\frac{1}{2}r)^{d_0}\,P^{(\frac{d}{2},\frac{d_0}{2})}_{l-1}(\cos r)\, .
\end{align}
The last formula in \eqref{eqB10} can be extracted from the tables 
\cite[Sec.7.391, Eq.(11)]{01} or derived directly, using Rodrigues' formula 
\eqref{eq9.11}.

Using the formula \eqref{eqC10***} together with \eqref{eqA10*} and 
\eqref{eqB10}, we find that
\begin{align}
	%\label{eqB10*}
\mu (\BBB(x_1,r)\cap \BBB(x_2,r))&=
&v(r)^2+\kappa \, \sum\nolimits_{l\ge 1}\, l^{-2}\, M_l\, a_l(r)\, 
\phi_l(\cos\theta(x_1,x_2))\, ,
\label{eqA.16}
\end{align}
where
\begin{equation}
	a_l(r)=
	\left\{ P^{(\frac{d}{2},\frac{d_0}{2})}_{l-1} (\cos r)\right\}^2  
	(\sin\frac 12 r)^{2d}(\cos \frac12 r)^{2d_0} 
	\, \dd \xi (r). 
	\label{eq9.16**}
\end{equation}
Substituting \eqref{eqA.16} into \eqref{eqA.16*}, we obtain the expansion
\begin{equation}
	\theta^{\Delta}(\xi,x_1,x_2) 
	=\langle \theta^{\Delta}(\xi) \rangle  
	-\kappa \, \sum\nolimits_{l\ge 1}\, l^{-2}\, M_l\, A_l(\xi)\, 
	\phi_l(\cos\theta(x_1,x_2))\, ,
	\label{eqA.16**}
	\end{equation}
where $\langle \theta^{\Delta}(\xi) \rangle  
=\int_{0}^{\pi}\Big(v(r)-v(r)^2\Big)\, \dd\xi (r)$	is the average value of the
metric 
%$\theta^{\Delta} (\xi)$ 
and $A_l(\xi)=\int_{0}^{\pi}\,a_l(r)\dd r$ are given in \eqref{eq9.16*}.
	%\end{proof}
Since $\theta^{\Delta}(\xi)$ is a metric,  we put $x_1=x_2$ to obtain
	$
	\langle \theta^{\Delta}(\xi) \rangle = 
	\kappa \, \sum\nolimits_{l\ge 1}\, l^{-2}\, M_l\, A_l(\xi)\, .
	$	
Substituting this equality into \eqref{eqA.16**}, we arrive to the expansion 
\eqref{eq9.16}.

The series \eqref{eqA.16} and \eqref{eq9.16} converge absolutely and uniformly 
in view of \eqref{eqC10****}.
\end{proof}

\begin{proof}[2.3 Proof of Lemma~1.1.]
Let us consider the embedding \eqref{eqC7} for $l=1$. From the formula 
\eqref{eq8.27*} we find
\begin{equation}
m_1=\frac{d(d+d_0+2)}{2d_0}	=\frac{(n+1)(d+2)}{2} -1,\quad d=nd_0,
	\label{eq9.16***}
\end{equation}
and for $x_1, x_2\in \QQQ$, we have
\begin{equation}\label{eqC7***}
	\Vert\Pi_1(x_1)-\Pi_1(x_2)\Vert^2=2-2(\Pi_1(x_1),\Pi_1(x_2))
	=2(1-\phi_1(\cos\theta(x_1,x_2)),
	\end{equation}
where $\Vert\cdot\Vert$ and $(\cdot,\cdot)$ are the Euclidean norm and inner 
product 
in $\Rr^{m_1}$.

On the other hand, from Rodrigues' formula \eqref{eq9.11} we obtain
\begin{equation*}\label{eqC7**}
P^{(\frac{d}{2} -1,\frac{d_0}{2} -1)}_1(t) =((d+d_0)t+d-d_0)/4,		
\end{equation*}
$P^{(\frac{d}{2} -1,\frac{d_0}{2} -1)}_1(1)=d/2$, and
\begin{equation*}
	%\label{eqC7**}
\frac{1-t}{2}=\frac{d}{d+d_0}\left[1-
\frac{P^{(\frac{d}{2} -1,\frac{d_0}{2} -1)}_1(t)}{P^{(\frac{d}{2} 
-1,\frac{d_0}{2} -1)}_1(1)}\right].
\end{equation*}
Therefore,
\begin{equation}\label{eqC7**}
\tau (x_1,x_2)^2=
\frac{1-\cos\theta(x_1,x_2)}{2}=\frac{d}{d+d_0}\Bigl[1-
\phi_1(\cos\theta(x_1,x_2))\Bigr].		
\end{equation}
Comparing 
\eqref{eqC7***} and \eqref{eqC7**}, we complete the proof.
\end{proof}

\begin{proof}[2.4 Proof of Theorem~1.2.]
Comparing the expansions \eqref{eq9.16**} and \eqref{eq9.16}, we conclude
that the equality \eqref{eq2.8}
is equivalent to the series of formulas
\begin{align}
\gamma(Q)\,l^{-2}\, B(d/2, d_0/2)^{-1}\, A_l(\xi^{\natural})= C_l/2\, ,
\quad l\ge 1\, . 
\label{eq0.03}
\end{align}

The integral \eqref{eq9.16*} with the special measure 
$\dd\xi^{\natural} (r)=\sin r\,\dd r$ takes the form
\begin{align}
A_l(\xi^{\natural})&=\int^{\pi}_{0}
\left\{ P^{(\frac{d}{2},\frac{d_0}{2})}_{l-1} (\cos r)\right\}^2  
(\sin\frac12 r)^{2d}(\cos \frac12 r)^{2d_0} 
\,\sin r\, \dd r
\notag
\\
&=2^{-d-d_0}\int^1_{-1}
\left(P^{(d/2, d_0/2)}_{l-1}(t)\right)^2 \,\left(1-t\right)^{d}
\left(1+t\right)^{d_0}\, \dd t \, .
\label{eq0.05}
\end{align}
Hence, the formulas \eqref{eq0.03} can be written as follows
\begin{align}
\int^1_{-1}
\left(P^{(d/2, d_0/2)}_{l-1}(t)\right)^2 \,
\left(1-t\right)^{d}&
\left(1+t\right)^{d_0}\, \dd t 
\notag
\\
=\, &\frac{2^{d+d_0+1}\, (1/2)_{l-1}}{((l-1)!)^2}\,B(d+1, d_0 +1)\, T^*,
\label{eq0.03*}
\end{align}
where
\begin{equation}
T^*=\frac{(l!)^2\, B(d/2, d_0/2)\,\, C_l}
{4\,(1/2)_{l-1}\,B(d+1, d_0 +1)\,\gamma (Q)}\, .
\label{eq0.06*}
\end{equation}

On the other hand, using \eqref{eq8.23**} and \eqref{eq8.26*}, we find 
\begin{equation}
C_l=(l!)^{-1}\, (1/2)_{l-1}\frac{\Gamma (d/2+1/2)\,\Gamma (l+d/2)\,\Gamma 
	(l+d_0/2)}
{\Gamma (l+1/2+d/2+d_0/2)\,\Gamma (d/2)}\, .
\label{eq0.07}
\end{equation}
Substituting \eqref{eq0.07} and \eqref{eq1.33*} into \eqref{eq0.06*}, we obtain
\begin{align}
T^*=
& \pi^{-1/2}\,(d+d_0)^{-1}\,\frac{\Gamma (d+d_0 +2)}
{\Gamma (d+1)\,\Gamma (d_0 +1)}\,\times
\notag
\\
\notag
\\
&\times
\frac{\Gamma (d/2+1/2)\,\Gamma (l+d/2)\,\Gamma (d_0/2+1/2)\,\Gamma (l+d_0/2)}
{\,\Gamma (d/2 +d_0/2)\,\Gamma (l+d/2 +d_0/2 +1/2)}\, .
\label{eq0.08}
\end{align}
Applying the duplication formula for the gamma function
%, see \cite[Sec.~1.5]{1*},
\begin{equation}
\Gamma (2z)=\pi^{-1/2}\,2^{2z-1}\, \Gamma (z)\,\Gamma(z+1/2) 
\label{eq0.09}
\end{equation}
to the first co-factor in \eqref{eq0.08}, we find  
\begin{align}
\pi^{-1/2}\,(d+d_0)^{-1}\,&\frac{\Gamma (d+d_0 +2)}
{\Gamma (d+1)\,\Gamma (d_0 +1)}
%\notag
%\\
\notag
\\
=\,\,&\frac{\Gamma (d/2 +d_0/2)\,\Gamma (d/2 +d_0/2 +3/2)}
{\Gamma (d/2 +1/2)\,\Gamma (d_0/2 +1)\,\Gamma (d_0/2 +1/2)\,\Gamma (d_0/2 
	+1)}\, ,
\label{eq0.010}
\end{align}
where the relation $\Gamma (z+1) = z\Gamma (z)$ with $z=d/2+d_0/2$
has been also used.

Substituting \eqref{eq0.010} into \eqref{eq0.08}, we find that 
$T^*=T_{l-1}(d/2, d_0/2)$.	Replacing $l-1,\, l\ge 1,$ with $l\ge 0$, we 
compete 
the proof.
\end{proof}

\section{Proof of Lemma~1.2}\label{sec6}

Lemma~1.2 follows from Lemma~3.1 and Lemma~3.2 given below.
\begin{lemma}\label{lem4.1}
For all $l\ge 0$, $\Re\alpha >-1/2$ and $\,\Re\beta >-1/2$, we have
\begin{align}
\int^1_{-1}\left(P^{(\alpha, \beta)}_{l}(t)\right)^2& 
(1-t)^{2\alpha}
(1+t)^{2\beta}\, \dd t 
\notag
\\
& 
=\,\frac{2^{2\alpha+2\beta+1}}
{(l!)^2}\,B(2\alpha+1, 2\beta +1)\,
\frac{W_l(\alpha, \beta)}{(2\alpha +2\beta +2)_{2l}} \, ,
%V_n(\alpha, \beta),
\label{eq0.011}
\end{align}
where
\begin{align}
W_l(\alpha, \beta&)
\notag
\\
=&\sum\nolimits_{k=0}^{2l}\frac{(-1)^{l+k}}{k!} 
\langle 2l \rangle_k \,\langle \alpha +l \rangle_k \,\langle \beta +l 
\rangle_{2l-k}\,
(2\alpha +1)_{2l-k} \, (2\beta +1)_k  
\label{eq0.012}
\end{align} 
is a polynomial of $\alpha$ and $\beta$.
\end{lemma}
\begin{proof}[Proof.]
Using Rodrigues' formula \eqref{eq9.11}, we can write
\begin{align}
\int^1_{-1}\left(P^{(\alpha, \beta)}_{l}(t)\right)^2 
(1-t)^{2\alpha}
(1+t)^{2\beta}\, \dd t =  \Big (\, \frac{1}{2^l\,l!}\,\Big)^2 \, I_l(\alpha, 
\beta)\,.
\label{eq0.011a}
\end{align}
where
\begin{align}
I_l(\alpha, \beta)=\int^1_{-1}\Big(\frac{\dd^l}{\dd t^l} 
\left [ (1-t)^{l+\alpha} (1+t)^{l+\beta} \right ] \Big)^2
\, \dd t \,.
\label{eq0.012a}
\end{align}
Integrating in \eqref{eq0.012a} $l$ times by parts, we obtain
\begin{align}
I_l(&\alpha, \beta)
\notag
\\
&=(-1)^l\,\int^1_{-1}\left( (1-t)^{l+\alpha} (1+t)^{l+\beta} \right)\,
\frac{\dd^{2l}}{\dd t^{2l}} \left( (1-t)^{l+\alpha} (1+t)^{l+\beta} \right)  
\, \dd t \,,
\label{eq0.013a}
\end{align}
since all terms outside the integral vanish.
By Leibniz's rule,
\begin{align*}
\frac{\dd^{2l}}{\dd t^{2l}}& \left(  (1-t)^{l+\alpha}\, (1+t)^{l+\beta} \right) 
\notag
\\ 
&=\sum\nolimits_{k=0}^{2l}\, {2l\choose k}\,\,
\frac{\dd^k}{\dd t^k} (1-t)^{l+\alpha}\,\,
\frac{\dd^{2l-k}}{\dd t^{2l-k}}\, (1+t)^{l+\beta} \, ,
%\label{eq0.013aa}
\end{align*}
where ${2l\choose k}=\langle 2l \rangle_k /k!\,$ and 
\begin{align*}
\frac{\dd^k}{\dd t^k} (1-t)^{l+\alpha}=(-1)^k\,\langle \alpha +l \rangle_k\, 
(1-t)^{l-k+\alpha} \, ,
\\        
\frac{\dd^{2l-k}}{\dd t^{2l-k}} (1+t)^{l+\beta}=\langle \beta +l 
\rangle_{2l-k}\, (1+t)^{-l+k+\beta}\, .
%\label{eq0.012aa}
\end{align*}
Substituting these formulas into \eqref{eq0.013a}, we obtain
\begin{align}
I_l(\alpha, \beta&)
\notag
\\
=&\,2^{2\alpha +2\beta +2l+1}\sum\nolimits_{k=0}^{2l}\frac{(-1)^{l+k}}{k!} 
\langle 2l \rangle_k \,\langle \alpha +l \rangle_k \,\langle \beta +l 
\rangle_{2l-k}\,
\,I_l^{(k)}(\alpha ,\beta)\, ,  
\label{eq0.012aa}
\end{align} 
where
\begin{equation} 
I_l^{(k)}(\alpha ,\beta)=B(2\alpha +2l-k+1, 2\beta +k+1).
\label{eq0.0ab} 
\end{equation}
Here we have used the following Euler's integral 
\begin{align}
2^{1-a-b}\int_{-1}^{1} (1-t)^{a-1}\,(1+t)^{b-1}\,\dd t =
B(a,b)= \frac{\Gamma (a)\Gamma (b)}{\Gamma (a+b)}
\label{eq0.02}
\end{align}
with $\Re a>0,\, \Re b>0$. 
The formula \eqref{eq0.0ab} can be written as follows
\begin{align}
&I_l^{(k)}(\alpha, \beta)=
\frac{\Gamma (2\alpha +2l-k+1)\,\Gamma (2\beta +k+1)}
{\Gamma (2\alpha +2\beta +2l+2)}
\notag
\\
\notag
\\
=&\frac{\Gamma (2\alpha +2l-k+1)}{\Gamma (2\alpha +1)}
\frac{\Gamma (2\beta +k+1)}{\Gamma (2\beta +1)}
\frac{\Gamma (2\alpha +1)\,\Gamma (2\beta +1)}{\Gamma (2\alpha +2\beta +2)}
\frac{\Gamma (2\alpha +2\beta +2)}{\Gamma (2\alpha +2\beta +2l+2)}
\notag
\\
\notag
\\
=&\frac{(2\alpha +1)_{2l-k}\, (2\beta +1)_k}{(2\alpha +2\beta +2)_{2l}}\, 
B(2\alpha +1, 2\beta +1 )\, .
\label{eq0.014}
\end{align} 

Combining the formulas \eqref{eq0.014}, \eqref{eq0.012aa} and \eqref{eq0.011a},
we obtain \eqref{eq0.011}.
\end{proof}

The next Lemma~3.2 is more specific, it relies on 
Watson's theorem for generalized hypergeometric series, see \cite{1*, 33a}.
We consider the series of the form 
\begin{align}
_3F_2(a, b, c; d, e; z)=
%\notag
%\\
\sum\nolimits_{k\ge 0}\frac{(a)_k\, (b)_k\, (c)_k\,}{(d)_k\, (e)_k\, k!}\, z\, 
,  
\label{eq0.015}
\end{align}
where neither $d$ nor $e$ are negative integers.  
The series absolutely converges for $\vert z\vert\le 1$, if 
$\Re (d+e)>\Re (a+b+c)$. 
The series \eqref{eq0.015} terminates, if one of the numbers  $a, b, c$ 
is a negative integer. 

\bfseries Watson's theorem.\mdseries {\it We have}
\begin{align}
_3F_2(a, &b, c; (a+b+1)/2, 2c; 1)
\notag
\\
=&\frac{\Gamma (1/2)\,\Gamma (c+1/2)\,\Gamma ((a+b+1)/2)\,\Gamma 
(c-(a+b-1)/2)\,}
{\Gamma ((a+1)/2)\,\Gamma ((b+1)/2)\,\Gamma (c-(a-1)/2)\,\Gamma (c-(b-1)/2)\,} 
\, .  
\label{eq0.016}
\end{align}  
{\it provided that}  
\begin{equation}
\Re\, (2c - a -b+1) > 0. 
\label{eq0.017}
\end{equation}

The condition \eqref{eq0.017} ensures the convergence of hypergeometric series 
in \eqref{eq0.016}. Furthermore, this condition is necessary for the truth 
of equality \eqref{eq0.016} even in the case of terminated series. The proof of
Watson's theorem can be found in \cite[Therem~3.5.5]{1*}, 
\cite[p.54, Eq.(2.3.3.13)]{33a}.

\begin{lemma}\label{lem4.2}
For all $l\ge 0$, $\alpha \in \Cc$ and $\beta\in\Cc$,	
the polynomial \eqref{eq0.012} is equal to
\begin{align}
W_l(\alpha, \beta)=&2^{2l}\,(\alpha +1)_l\,(\beta +1)_l\,(\alpha +\beta +1)_l
\notag
\\
=&2^{2l}\,\frac{\Gamma(\alpha +1+l)\,\Gamma(\beta +1+l)\,\Gamma(\alpha +\beta 
+1+l)}
{\Gamma(\alpha +1)\,\Gamma(\beta +1)\,\Gamma(\alpha +\beta +1)}\, .
\label{eq0.018}
\end{align} 
In particular, 
\begin{align}
\frac{W_l(\alpha, \beta)}{(2\alpha +2\beta +2)_{2l}}=
\frac{(\alpha +1)_l\, (\beta +1)_l}{(\alpha +\beta +3/2)_l}\, .
\label{eq0.018a}
\end{align} 
\end{lemma}
\begin{proof}[Proof.]
Since $W_l(\alpha, \beta)$ is a polynomial, it suffers to check 
the equality \eqref{eq0.018} for $\alpha$ and $\beta$ in an open subset in $\Cc 
^2$. As such a subset we will take the following region
\begin{equation} 
\O=\{\,\alpha,\,\beta\, : \Re\alpha <0,\,\,\Re\beta <0, \,\, \Im\alpha 
>0,\,\,\Im\beta >0\,\}.
\label{0.019}   
\end{equation} 
For $\alpha$ and $\beta$ in $\O$, the co-factors in terms in \eqref{eq0.012}
may be rearranged as follows: 
\begin{equation}
\left.
\begin{aligned}
&\langle 2l \rangle_k =(-1)^k\,(-2l)_k \, ,\quad
\langle \alpha +l \rangle_k = (-1)^k\, (-\alpha -l)_k\, , \\
&\langle \beta +l \rangle_{2l-k}=(-1)^k\,(-\beta -l)_{2l-k}=
\frac{(-\beta -l)_{2l}}{(\beta +1-l)_k}\, ,\quad \\
&(2\alpha +1)_{2l-k} = \frac{(-1)^k(2\alpha +1)_{2l}}{(-2\alpha -2l)_k}\, ,
\end{aligned}
\label{eq1.34a}
%\end{align}
\right\}
\end{equation}
Here we have used the following elementary relation for the rising factorial 
powers
\begin{equation}
(a)_{m-k}=\frac{(-1)^k\,(a)_m}{(1-a-m)_k}\, ,\quad m\ge 0\, ,\,\, 0\le k \le m 
\, .
\end{equation}
Substituting \eqref{eq1.34a} into \eqref{eq0.012}, we find that 
\begin{align}
W_l(\alpha, \beta)&=
%\notag
%\\
(-1)^l\,(2\alpha +1)_{2l}\,(-\beta -l)_{2l}\,\FFF_l(\alpha , \beta )
\notag
\\
&=\frac{(-1)^l\,\Gamma (2\alpha +1+2l)\,\Gamma (-\beta +l)}
{\Gamma (2\alpha +1)\,\Gamma (-\beta -l)}\,\, \FFF_l(\alpha , \beta )\, ,
\label{eq0.35a}
\end{align} 
where 
\begin{equation} 
\FFF_l(\alpha , \beta )=\sum\nolimits_{k=0}^{2l}\frac
{(-2l)_k\,(2\beta +1)_k\,(-\alpha -l)_k}
{(\beta +1-l)_k\,(-2\alpha -2l)_k\,k!}
\label{eq0.355}
\end{equation}
In view of the definition \eqref{eq0.015}, we have
\begin{align}
\FFF_l(\alpha , \beta )=\,
%{\mathstrut }
_3F _2\,(-2l, 2\beta +1, -\alpha -1; \beta 
+1-l, -2\alpha -2l; 1)\, .
\label{eq0.40}
\end{align} 
The parameters in hypergeometric series  
\eqref{eq0.40} are identical  with those in \eqref{eq0.016} for
%\begin{equation}  
$a=-2l, \, b=2\beta +1, \, c= -\alpha -l$,
%\end{equation}
and in this case, $(a+b+1)/2 = 2\beta +1 +l$, $2c = -2\alpha -2l$.
The condition \eqref{eq0.017} also holds for $\alpha$ and $\beta$ 
in the region $\O$, since $\Re\, (2c - a -b+1)=\Re\, (-2\alpha -2\beta ) >0$.
Therefore, Watson's theorem \eqref{eq0.016} can be applied to obtain 
\begin{align}
\FFF_l(\alpha , \beta )=\,
\frac{\Gamma(1/2)\,\Gamma(-\alpha -l-1/2)\,\Gamma(\beta +1-l)\,
\Gamma(-\alpha -\beta)}
{\Gamma(-l+ 1/2)\,\Gamma(\beta +1)\,\Gamma(-\alpha +1/2)\,
\Gamma(-\alpha -\beta -l)}
\, .
\label{eq0.41}
\end{align} 

Substituting the expression \eqref{eq0.41} into \eqref{eq0.35a}, we may write
\begin{equation}
W_l(\alpha, \beta )=\, c_0\,\, c_1(\alpha )\,\, c_2(\beta )\,\,c_3(\alpha 
+\beta )\, ,
\label{eq0.42}
\end{equation}
where
\begin{equation}
\left.
\begin{aligned}
&c_0 = \frac{(-1)^l\, \Gamma (1/2)}{\Gamma (-l+1/2)}\, ,\\ 
%\\
&c_1(\alpha)=\frac{\Gamma (2\alpha +2l+1)\,\Gamma (-\alpha -l+1/2)}
{\Gamma (2\alpha +1)\,\Gamma (-\alpha +1/2)}\,,\quad \\
%\\
&c_2(\beta)= \frac{\Gamma (\beta +1-l)\,\Gamma (-\beta +l)}
{\Gamma (\beta +1)\,\Gamma (-\beta -l)}\, ,\\
%\\
&c_3(\alpha +\beta)=\frac{\Gamma (-\alpha -\beta )}
{\Gamma (-\alpha -\beta -l)}\, .
\end{aligned}
\label{eq1.45}
\right\}
\end{equation}

Using the duplication formula \eqref{eq0.09} and reflection formulas,
see \cite[Sec.~1.2]{1*}, 
\begin{equation}
\Gamma (1-z)\Gamma (z)\, =\, \frac{\pi}{\sin \pi z}\, , \qquad 
\Gamma (1/2 -z)\Gamma (1/2 +z)\, =\, \frac{\pi}{\cos \pi z} \, ,
\label{eq0.43}    
\end{equation}
we may rearrange the expressions in \eqref{eq1.45} as follows.
For $c_0$, we have
\begin{align*}
&c_0 = \frac{(-1)^l\,\Gamma (1/2)^2}{\Gamma (-l+1/2)\,\Gamma (l+1/2)}\, 
\frac{\Gamma (l+1/2)}{\Gamma (1/2)}=(1/2)_l \, ,
\end{align*} 
since $\Gamma(1/2)=\sqrt\pi $.
For $c_1(\alpha)$ and $c_2(\beta)$, we have
\begin{align*}
c_1(\alpha)=&2^{2l}\,
\frac{\Gamma(\alpha +l+1)\,\Gamma(\alpha +l+1/2)\,\Gamma(-\alpha -l+1/2)}
{\Gamma(\alpha +1)\,\Gamma(\alpha +1/2)\,\Gamma(-\alpha +1/2)}
\\
\\
=&2^{2l}\,
\frac{\cos\pi\alpha\,\Gamma(\alpha +l+1)}{\cos\pi(\alpha+l)\,\Gamma(\alpha +1)}
=2^{2l}\,(-1)^l\,(\alpha +1)_l
\end{align*}
and
\begin{align*}
c_2(\beta) = \frac{\Gamma (\beta +1-l)\,\Gamma (-\beta +l)}
{\Gamma (\beta +1)\,\Gamma (-\beta -l)}
%\\
=\frac{\sin\pi (\beta +l)\,\Gamma (\beta +1+l)}
{\sin\pi (\beta -l)\,\Gamma (\beta +1)}=(\beta +1)_l \, .
\end{align*}
Finally,
\begin{align*}
c_3(\alpha +\beta)=\frac{\sin\pi (\alpha +\beta)\,\Gamma (\alpha +\beta 
+1+l)}
{\sin\pi (\alpha +\beta +l)\,\Gamma (\alpha +\beta +1)}=(-1)^l\, 
(\alpha +\beta +1)_l
\, .
\end{align*}
Substituting these expressions into \eqref{eq0.42}, we obtain \eqref{eq0.018}.

It follows from \eqref{eq8.26*} and the duplication formula \eqref{eq0.09}
that 
\begin{equation}
(2\alpha +2\beta +2)_{2l}=
2^{2l}\,(\alpha +\beta +1)_l\, (\alpha +\beta +3/2)_l \, .
\label{eq0.47}
\end{equation}
Using \eqref{eq0.018} together with \eqref{eq0.47}, we obtain \eqref{eq0.018a}.
\end{proof}
Now it suffers to substitute \eqref{eq0.018a} into \eqref{eq0.011} 
to obtain  the formulas \eqref{eq01}. The proof of Lemma~1.2 
is complete.

%%%%%%%%%%%%%%%%%%%%%%%%%%%%%%%%%%%%%%%%%%%%%%%%%
%%%%%%%%%%%%%%%%%%%%%%%%%%%%%%%%%%%%%%%%%%%%%%%%%%%

\end{document}